\documentclass[preprint]{imsart}

\RequirePackage{amsthm,amsmath,amsfonts,amssymb}
\RequirePackage[numbers]{natbib}
\RequirePackage[colorlinks,citecolor=blue,urlcolor=blue]{hyperref}
\RequirePackage{graphicx}

\startlocaldefs
\endlocaldefs

\newtheorem{thm}{Theorem}

\newtheorem{cor}[thm]{Corollary}

\theoremstyle{definition}

\newtheorem{ex}[thm]{Example}

\providecommand{\abs}[1]{\lvert#1\rvert}
\providecommand{\Abs}[1]{\Bigl\lvert#1\Bigr\rvert}
\providecommand{\norm}[1]{\lVert#1\rVert}

\begin{document}

\begin{frontmatter}
\title{Bayesian predictive inference\\without a prior}
%\title{A sample article title with some additional note\thanksref{t1}}
\runtitle{Bayesian predictive inference}
%\thankstext{T1}{A sample additional note to the title.}

\begin{aug}
%%%%%%%%%%%%%%%%%%%%%%%%%%%%%%%%%%%%%%%%%%%%%%
%%Only one address is permitted per author. %%
%%Only division, organization and e-mail is %%
%%included in the address.                  %%
%%Additional information can be included in %%
%%the Acknowledgments section if necessary. %%
%%%%%%%%%%%%%%%%%%%%%%%%%%%%%%%%%%%%%%%%%%%%%%
\author[A]{\fnms{Patrizia} \snm{Berti}\ead[label=e1]{patrizia.berti@unimore.it}}
\author[B]{\fnms{Emanuela} \snm{Dreassi}\ead[label=e2]{emanuela.dreassi@unifi.it}}
\author[C]{\fnms{Fabrizio} \snm{Leisen}\ead[label=e3]{fabrizio.leisen@gmail.com}}
\author[D]{\fnms{Luca} \snm{Pratelli}\ead[label=e4]{pratel@mail.dm.unipi.it}}
\and
\author[E]{\fnms{Pietro} \snm{Rigo}\ead[label=e5]{pietro.rigo@unibo.it}}
%%%%%%%%%%%%%%%%%%%%%%%%%%%%%%%%%%%%%%%%%%%%%%
%% Addresses                                %%
%%%%%%%%%%%%%%%%%%%%%%%%%%%%%%%%%%%%%%%%%%%%%%
\address[A]{Dipartimento di Matematica Pura ed Applicata ``G. Vitali'', Universit\`a di Modena e Reggio-Emilia, via Campi 213/B, 41100 Modena, Italy
}

\address[B]{Dipartimento di Statistica, Informatica, Applicazioni, Universit\`a di Firenze, viale Morgagni 59, 50134 Firenze, Italy
}

\address[C]{School of Mathematical Sciences, University of Nottingham, University Park, Nottingham, NG7 2RD, UK
}

\address[D]{Accademia Navale, viale Italia 72, 57100 Livorno,
Italy
}

\address[E]{Dipartimento di Scienze Statistiche ``P. Fortunati'', Universit\`a di Bologna, via delle Belle Arti 41, 40126 Bologna, Italy
}
\end{aug}

\begin{abstract}
Let $(X_n:n\ge 1)$ be a sequence of random observations. Let $\sigma_n(\cdot)=P\bigl(X_{n+1}\in\cdot\mid X_1,\ldots,X_n\bigr)$ be the $n$-th predictive distribution and $\sigma_0(\cdot)=P(X_1\in\cdot)$ the marginal distribution of $X_1$. In a Bayesian framework, to make predictions on $(X_n)$, one only needs the collection $\sigma=(\sigma_n:n\ge 0)$. Because of the Ionescu-Tulcea theorem, $\sigma$ can be assigned directly, without passing through the usual prior/posterior scheme. One main advantage is that no prior probability has to be selected. In this paper, $\sigma$ is subjected to two requirements: (i) The resulting sequence $(X_n)$ is conditionally identically distributed, in the sense of \cite{BPR2004}; (ii) Each $\sigma_{n+1}$ is a simple recursive update of $\sigma_n$. Various new $\sigma$ satisfying (i)-(ii) are introduced and investigated. For such $\sigma$, the asymptotics of $\sigma_n$, as $n\rightarrow\infty$, is determined. In some cases, the probability distribution of $(X_n)$ is also evaluated.
\end{abstract}

\begin{keyword}[class=MSC2010]
\kwd[Primary ]{62F15}
\kwd{62M20}
\kwd{60G25}
\kwd[; secondary ]{60G09}
\end{keyword}

\begin{keyword}
\kwd{Asymptotics}
\kwd{Bayesian nonparametrics}
\kwd{Conditional identity in distribution}
\kwd{Exchangeability}
\kwd{Predictive distribution}
\kwd{Sequential prediction}
\kwd{Total variation distance}
\end{keyword}

\end{frontmatter}

\maketitle

\section{Introduction}\label{intro}

Consider a Bayesian forecaster who makes predictions on a sequence $(X_n:n\ge 1)$ of random observations. At each time $n$, she aims to predict $X_{n+1}$ based on $(X_1,\ldots,X_n)$. To this end, she needs to assign the conditional distribution of $X_{n+1}$ given $(X_1,\ldots,X_n)$, usually called the $n$-th {\em predictive distribution}.

To formalize this problem, we fix a measurable space $(S,\mathcal{B})$ and we take $X_n$ to be the $n$-th coordinate random variable on $S^\infty$, namely
\begin{gather*}
X_n(s_1,\ldots,s_n,\ldots)=s_n
\end{gather*}
for each $n\ge 1$ and each $(s_1,\ldots,s_n,\ldots)\in S^\infty$. Moreover, following Dubins and Savage \cite{DS1965}, we introduce the notion of strategy.

Let $\mathcal{P}$ denote the collection of all probability measures on $\mathcal{B}$. A {\em strategy} is a sequence $\sigma=(\sigma_0,\sigma_1,\ldots)$ such that

\vspace{0.2cm}

\begin{itemize}

\item $\sigma_0\in\mathcal{P}$ and $\sigma_n=\{\sigma_n(x):x\in S^n\}$ is a collection of elements of $\mathcal{P}$;

\vspace{0.2cm}

\item The map $x\mapsto\sigma_n(x)(A)$ is $\mathcal{B}^n$-measurable for fixed $n\ge 1$ and $A\in\mathcal{B}$.

\end{itemize}

\vspace{0.2cm}

\noindent Here, $\sigma_0$ should be regarded as the marginal distribution of $X_1$ and $\sigma_n(x)$ as the conditional distribution of $X_{n+1}$ given that $(X_1,\ldots,X_n)=x$. Moreover, $\sigma_n(x)(A)$ denotes the probability attached to the event $A$ by the probability measure $\sigma_n(x)$.

For any strategy $\sigma$, there is a unique probability measure $P_\sigma$ on $(S^\infty,\mathcal{B}^\infty)$ such that
\begin{gather*}
P_\sigma(X_1\in\cdot)=\sigma_0\quad\text{and}\quad P_\sigma\bigl(X_{n+1}\in\cdot\mid (X_1,\ldots,X_n)=x\bigr)=\sigma_n(x)
\\\text{for all }n\ge 1\text{ and }P_\sigma\text{-almost all }x\in S^n\notag.
\end{gather*}

The above result, due to Ionescu-Tulcea, provides the theoretical foundations of Bayesian predictive inference. To make predictions on $(X_n)$, one needs precisely a strategy $\sigma$. The Ionescu-Tulcea theorem guarantees that, for {\em any} $\sigma$, the predictions based on $\sigma$ are consistent with a unique probability distribution $P_\sigma$ for the data sequence $(X_n)$.

However, $(X_n)$ is usually required some distributional properties suggested by the specific problem under consideration. For instance, $(X_n)$ is asked to be exchangeable, or stationary, or Markov, and so on. In these cases, the strategy $\sigma$ can not be arbitrary, for $P_\sigma$ must belong to some given class of probability measures on $(S^\infty,\mathcal{B}^\infty)$.

\subsection{Motivations}\label{exc5} In a Bayesian framework, $(X_n)$ is typically assumed to be {\em exchangeable}. In that case, there are essentially two approaches for selecting a strategy $\sigma$. For definiteness, as in \cite{BDPR2021}, we call them the {\em standard approach} (SA) and the {\em non-standard approach} (NSA). Both are admissible from the Bayesian point of view and both lead to a full specification of the probability distribution of $(X_n)$.

According to SA, to obtain $\sigma$, one should:

\vspace{0.2cm}

\begin{itemize}

\item Select a prior $\pi$, namely, a probability measure on $\mathcal{P}$;

\vspace{0.2cm}

\item Calculate the posterior of $\pi$ given that $(X_1,\ldots,X_n)=x$, say $\pi_n(x)$;

\vspace{0.2cm}

\item Evaluate $\sigma$ as
\begin{gather*}
\sigma_n(x)(A)=\int_\mathcal{P}p(A)\,\pi_n(x)(dp)\quad\text{for all }A\in\mathcal{B},
\end{gather*}
where $\pi_0(x)$ is meant as  $\pi_0(x)=\pi$.
\end{itemize}

\vspace{0.2cm}

Instead, according to NSA, the strategy $\sigma$ can be assigned directly, without passing through the above prior/posterior scheme. Rather than choosing $\pi$ and evaluating $\pi_n$ and $\sigma_n$, the forecaster merely selects her predictive $\sigma_n$. This procedure makes sense because of the Ionescu-Tulcea theorem. See e.g. \cite{BRR1997}, \cite{BCPR2009}, \cite{BDPR2021}, \cite{BDPR2021BIS}, \cite{CR1996}, \cite{FHW2021}, \cite{FLR2000}, \cite{FP2012}, \cite{FP2020}, \cite{HMW}, \cite{HILL}, \cite{QMT13}, \cite{PIT1996}, \cite{PIT06}.

The merits and drawbacks of SA and NSA are discussed in \cite{BDPR2021}. In short, SA is a cornerstone of Bayesian inference but is not motivated by prediction alone. Its main scope is to make inference on other features of the data distribution, such as a random parameter (possibly, infinite dimensional). However, when prediction is the main target, SA is clearly involved. In turn, NSA has essentially two merits. Firstly, it requires the assignment of probabilities on {\em observable facts} only. The next observation $X_{n+1}$ is actually observable, while $\pi$ and $\pi_n$ (being probabilities on $\mathcal{P}$) do not deal with observable facts. Secondly, and more importantly, NSA is much more efficient than SA when prediction is the main goal. In this case, why select the prior $\pi$ explicitly ? Rather than wondering about $\pi$, it seems reasonable to reflect on how $X_{n+1}$ is affected by $(X_1,\ldots,X_n)$.

The above remarks refer to any (Bayesian) prediction problem, both parametric and nonparametric. However, NSA is especially appealing in the {\em nonparametric} case, where selecting a prior with large support is usually hard. For instance, NSA is quite natural when dealing with species sampling sequences. Indeed, this paper has been written having the nonparametric framework in mind.

If $(X_n)$ is assumed to be exchangeable, however, NSA has a gap. Given an arbitrary strategy $\sigma$, the Ionescu-Tulcea theorem does not grant exchangeability of $(X_n)$ under $P_\sigma$. Therefore, for NSA to apply, one should first characterize those strategies $\sigma$ which make $(X_n)$ exchangeable under $P_\sigma$. A nice characterization is \cite[Th. 3.1]{FLR2000}. However, the conditions on $\sigma$ for making $(X_n)$ exchangeable are quite hard to check in real problems.

To bypass the gap mentioned in the above paragraph, the exchangeability assumption could be weakened. One option is to assume $(X_n)$ to be {\em conditionally identically distributed} (c.i.d.). We refer to Subsection \ref{b7y} for c.i.d. sequences. Here, we just mention a few reasons for taking c.i.d. data into account.

\vspace{0.2cm}

\begin{itemize}

\item Roughly speaking, $(X_n)$ is c.i.d. if, at each time $n$, the future observations $(X_k:k>n)$ are identically distributed given the past $(X_1,\ldots,X_n)$. Hence, even if weaker than exchangeability, conditional identity in distribution is a natural assumption for predictive problems.

\item The asymptotic behavior of c.i.d. sequences is very similar to that of exchangeable ones.

\item A meaningful part of the usual Bayesian machinery can be developed under the sole assumption that $(X_n)$ is c.i.d.; see \cite{FHW2021}.

\item A number of interesting strategies cannot be used if $(X_n)$ is exchangeable, but are available if $(X_n)$ is only required to be c.i.d.; see e.g. \cite{BDPR2021}. Furthermore, conditional identity in distribution is more reasonable than exchangeability in a few real problems. Examples occur in various fields, including clinical trials, generalized Polya urns, species sampling models and disease surveillance; see \cite{ACBLG2014}, \cite{BCL}, \cite{BPR2004}, \cite{CZGV}.

\item It is not hard to characterize the strategies $\sigma$ which make $(X_n)$ c.i.d. under $P_\sigma$; see Theorem \ref{c5ty6}. Therefore, unlike the exchangeable case, NSA can be easily implemented. This remark is fundamental for this paper.

\end{itemize}

\vspace{0.2cm}

\subsection{Kernels}

Before going on, one more definition is in order.

A {\em kernel} (or a {\em random probability measure}) on $(S,\mathcal{B})$ is a collection
\begin{gather*}
\alpha=\{\alpha(x):x\in S\}
\end{gather*}
such that $\alpha(x)\in\mathcal{P}$ for each $x\in S$ and the map $x\mapsto\alpha(x)(A)$ is measurable for fixed $A\in\mathcal{B}$.

As an example, suppose that $S=\mathbb{R}$ and $\mathcal{B}$ is the Borel $\sigma$-field. Denote by $\mathcal{N}(a,b)$ the Gaussian law on $\mathcal{B}$ with mean $a\in\mathbb{R}$ and variance $b>0$, i.e.
\begin{gather*}
\mathcal{N}(a,b)(A)=\int_A (2\,\pi\,b)^{-1/2}\exp\bigl\{(x-a)^2/2b\bigr\}\,dx\quad\text{for all }A\in\mathcal{B}.
\end{gather*}
Then,
\begin{gather*}
\alpha(x)=\mathcal{N}\bigl(f(x),\,g(x)\bigr)
\end{gather*}
is a kernel on $(S,\mathcal{B})$ provided $f$ and $g$ are measurable functions from $S$ into itself and $g>0$.

\subsection{Our contribution}\label{z3e5t} This paper is the natural follow up of \cite{BDPR2021} and aims to develop NSA for c.i.d. data. Our main goal is to introduce and investigate new strategies $\sigma$ having the following two properties:

\vspace{0.2cm}

\begin{itemize}

\item[(i)] The sequence $(X_n)$ is c.i.d. under $P_\sigma$;

\vspace{0.2cm}

\item[(ii)] $\sigma_{n+1}$ is a simple recursive update of $\sigma_n$ for each $n\ge 0$.

\end{itemize}

\vspace{0.2cm}

Condition (i) has been already discussed. Condition (ii) is to obtain a fast online Bayesian prediction, in the spirit of \cite{HMW}. Ideally, condition (ii) should imply that each predictive can be evaluated through a simple recursion on the previous one.

To make some examples, let us define
\begin{gather*}
(x,y)=(x_1,\ldots,x_n,y_1,\ldots,y_m)
\end{gather*}
whenever $x=(x_1,\ldots,x_n)\in S^n$ and $y=(y_1,\ldots,y_m)\in S^m$. In this notation, condition (ii) is well realized if $\sigma$ satisfies the recursive equation
\begin{gather}\label{f5ty7}
\sigma_{n+1}(x,y)=q_n(x)\,\sigma_n(x)+(1-q_n(x))\,\alpha_n(y)
\end{gather}
for all $n\ge 0$, $x\in S^n$ and $y\in S$, where $q_n:S^n\rightarrow [0,1]$ is any measurable function and $\alpha_n$ a kernel on $(S,\mathcal{B})$.

According to \eqref{f5ty7}, the predictive $\sigma_{n+1}(x,y)$ is a convex combination of the previous predictive $\sigma_n(x)$ and a new contribution $\alpha_n(y)$. Moreover, the kernel $\alpha_n$ is driven by the last observation $y$ but is not affected by $x$, while the weight $q_n$ depends on $x$ but not on $y$. An obvious interpretation is that, at time $n+1$, after observing $(x,y)$, the next observation is drawn from $\sigma_n(x)$ with probability $q_n(x)$ and from $\alpha_n(y)$ with probability $1-q_n(x)$.

Even if simple, this updating rule is able to model various real situations; see Examples \ref{s5rt}$-$\ref{ali9}. Moreover, to implement such rule, no prior probability on $\mathcal{P}$ is required. The forecaster has only to choose three objects: The marginal distribution of $X_1$ (i.e., $\sigma_0$), the weight $q_n$ of the convex combination, and the contribution $\alpha_n$ of the last observation.

In addition to (ii), $\sigma$ is required to satisfy condition (i). As shown in \cite{BDPR2021}, the latter condition is always true provided
\begin{gather*}
\alpha_n(y)=\delta_y
\end{gather*}
where $\delta_y$ denotes the unit mass at the point $y$. Indeed, some popular strategies admit representation \eqref{f5ty7} with $\alpha_n(y)=\delta_y$. Well known examples are Dirichlet sequences, Beta-GOS sequences, exponential smoothing and generalized Polya urns; see \cite{ACBLG2014}, \cite{BCL} and \cite[Sect. 4]{BDPR2021}. However, for an arbitrary kernel $\alpha_n$, condition (i) may fail. Therefore, in Theorem \ref{m9i8}, we give conditions for $(X_n)$ to be c.i.d. under $P_\sigma$. In particular, these conditions are satisfied whenever
\begin{gather}\label{u9kj7}
\alpha_n(\cdot)(A)=E_{\sigma_0}\bigl(1_A\mid\mathcal{G}_n\bigr),\quad\quad\text{a.s. with respect to }\sigma_0,
\end{gather}
for some filtration $\mathcal{G}_0\subset\mathcal{G}_1\subset\mathcal{G}_2\subset\ldots$ on $(S,\mathcal{B})$.

As an example, take $\mathcal{G}_n=\sigma(\mathcal{H}_n)$ where $\mathcal{H}_n\subset\mathcal{B}$ is a countable partition of $S$ and $\sigma_0(H)>0$ for all $H\in \mathcal{H}_n$. In this case, condition \eqref{u9kj7} implies
\begin{gather*}
\alpha_n(y)=\sum_{H\in\mathcal{H}_n}1_H(y)\,\sigma_0(\cdot\mid H)=\sigma_0\bigl[\cdot\mid H_n(y)\bigr]
\end{gather*}
where $H_n(y)$ is the only $H\in\mathcal{H}_n$ such that $y\in H$. Moreover, $\mathcal{G}_n\subset\mathcal{G}_{n+1}$ provided the partition $\mathcal{H}_{n+1}$ is finer than $\mathcal{H}_n$. With this choice of $\alpha_n$, several meaningful strategies satisfying (i)-(ii), including extensions of Dirichlet and exponential smoothing, can be easily manufactured.

As a further example, not having the form \eqref{f5ty7}, take $S=\mathbb{R}$ and fix any sequence $u_n$ of real numbers such that
\begin{gather*}
0=u_0<u_1<u_2<\ldots<1.
\end{gather*}
Define also $f_0(x)=0$,
\begin{gather*}
f_{n+1}(x,y)=\sqrt{\frac{u_{n+1}-u_n}{1-u_n}}\,\,y+\Bigl(1-\sqrt{\frac{u_{n+1}-u_n}{1-u_n}}\Bigr)\,f_n(x)
\end{gather*}
and
\begin{gather}\label{qq2}
\sigma_n(x)=\mathcal{N}\Bigl(f_n(x),\,1-u_n\Bigr)
\end{gather}
where $n\ge 0$, $x\in S^n$ and $y\in S$. Such a $\sigma$ is in line with our scopes. In fact, to evaluate $\sigma_{n+1}(x,y)$, it suffices to know the last observation $y$ and the mean of $\sigma_n(x)$. Hence, condition (ii) holds. As shown in Theorem \ref{t6m9u}, condition (i) holds true as well. It is also shown that $(X_n)$ is a Gaussian sequence, under $P_\sigma$, with mean 0, variance 1, and a known covariance structure.

The idea underlying \eqref{qq2} can be generalized in various ways. Among other things, the normal distribution can be replaced by any other symmetric stable law. For instance, the normal distribution could be replaced by the Cauchy distribution if heavier tails are regarded more suitable for prediction.

Finally, we focus on the asymptotics of $\sigma_n$ as $n\rightarrow\infty$. In fact, if $(S,\mathcal{B})$ is a standard Borel space, condition (i) implies
\begin{gather*}
P_\sigma\bigl(\sigma_n\rightarrow\mu\text{ weakly}\bigr)=1
\end{gather*}
for some random probability measure $\mu$ on $(S,\mathcal{B})$; see Subsection \ref{b7y}. Our results deal with $\mu$. We give conditions for $\mu\ll\sigma_0$ a.s., for $\mu$ to be degenerate a.s., and for $\norm{\sigma_n-\mu}\overset{a.s.}\longrightarrow 0$ where $\norm{\cdot}$ is total variation norm; see Theorems \ref{u9mh1} and \ref{z5t6yh8}.

\vspace{0.2cm}

\section{Preliminaries}

\subsection{Further notation and assumptions on $(S,\mathcal{B})$} Let $\lambda,\,\nu\in\mathcal{P}$. We write $\lambda\ll\nu$ to mean that $\lambda$ is {\em absolutely continuous} with respect to $\nu$, namely, $\lambda(A)=0$ whenever $A\in\mathcal{B}$ and $\nu(A)=0$. Moreover, $\lambda$ and $\nu$ are {\em singular} if $\lambda(A)=\nu(A^c)=0$ for some $A\in\mathcal{B}$, and $\lambda$ is {\em diffuse} if $\lambda(\{y\})=0$ for all $y\in S$.

We denote by $x$ a point of $S^n$ where $n\ge 0$ is an integer or $n=\infty$. In both cases, $x_i$ is the $i$-th coordinate of $x$. If $n=0$ and $\sigma$ is a strategy, $\sigma_0(x)$ is meant as $\sigma_0(x)=\sigma_0$. Moreover, if $x\in S^\infty$ and $f$ is any map on $S^n$, we write $f(x)$ to denote $f(x)=f(x_1,\ldots,x_n)$. In particular,
\begin{gather*}
\sigma_n(x):=\sigma_n(x_1,\ldots,x_n)\quad\quad\text{for all }x\in S^\infty.
\end{gather*}

Finally, from now on, $S$ is a Borel subset of a Polish space and $\mathcal{B}$ is the Borel $\sigma$-field on $S$.

\subsection{Conditional identity in distribution}\label{b7y}

C.i.d. sequences have been introduced in \cite{BPR2004} and \cite{K} and then investigated in various papers; see e.g. \cite{ACBLG2014}, \cite{BCL}, \cite{BCPR2009}, \cite{BPR2012}, \cite{BPR2013}, \cite{BDPR2021}, \cite{BDPR2021BIS}, \cite{CZGV}, \cite{FHW2021}, \cite{FPS2018}, \cite{FP2020}.

Let $\mathcal{G}_n=\sigma(X_1,\ldots,X_n)$, $\mathcal{G}_0$ the trivial $\sigma$-field, and let $P$ be a probability measure on $(S^\infty,\mathcal{B}^\infty)$. Say that $(X_n)$ is c.i.d. (or that $P$ is c.i.d.) if
\begin{equation*}
P\bigl(X_k\in\cdot\mid\mathcal{G}_n\bigr)=
P\bigl(X_{n+1}\in\cdot\mid\mathcal{G}_n\bigr)\quad\text{a.s. for all }k>n\geq 0.
\end{equation*}
Thus, at each time $n\ge 0$, the future observations $(X_k:k>n)$ are identically distributed given the past. This is actually weaker than exchangeability. Indeed, $(X_n)$ is exchangeable if and only if it is stationary and c.i.d.

The asymptotic behavior of c.i.d. sequences is similar to that of exchangeable ones. In fact, suppose $P$ is c.i.d. and define the empirical measures
\begin{equation*}
\mu_n(x)=\frac{1}{n}\,\sum_{i=1}^n\delta_{x_i}\quad\quad\text{for all }n\ge 1\text{ and }x\in S^\infty.
\end{equation*}
Define also
\begin{gather*}
\mu(x)=\lim_n\mu_n(x)\,\text{ if the limit exists and }\,\mu(x)=\delta_{x_1}\,\text{ otherwise,}
\end{gather*}
where $x\in S^\infty$ and the limit is meant as a weak limit of probability measures. Then, for every fixed $A\in\mathcal{B}$,
\begin{gather*}
\mu(x)(A)=\lim_n\mu_n(x)(A)\quad\quad\text{for }P\text{-almost all }x\in S^\infty.
\end{gather*}
As a consequence, for each $n\geq 0$ and $A\in\mathcal{B}$, one obtains
\begin{gather*}
E_P\bigl\{\mu(A)\mid\mathcal{G}_n\bigr\}=P\bigl(X_{n+1}\in A\mid\mathcal{G}_n\bigr)\text{ a.s.}
\end{gather*}
Thus, as in the exchangeable case, $P\bigl(X_{n+1}\in\cdot\mid\mathcal{G}_n\bigr)=E\bigl\{\mu(\cdot)\mid\mathcal{G}_n\bigr\}$. By martingale convergence, this implies
\begin{gather}\label{d5tni9}
P\bigl(X_{n+1}\in A\mid\mathcal{G}_n\bigr)=E_P\bigl\{\mu(A)\mid\mathcal{G}_n\bigr\}\overset{a.s.}\longrightarrow\mu(A)\quad\quad\text{for each }A\in\mathcal{B}.
\end{gather}

In addition, $(X_n)$ is asymptotically exchangeable, in the sense that the probability distribution of the shifted sequence $(X_n,X_{n+1},\ldots)$ converges weakly to an exchangeable probability measure $Q$ on $(S^\infty,\mathcal{B}^\infty)$. Furthermore, $Q=P$ on the sub-$\sigma$-field $\sigma(\mu)$ generated by $\mu$.

A c.i.d. probability measure $P$ is not completely determined by $\mu$; see \cite[Ex. 17]{BDPR2021}. Hence, the role played by $\mu$ is not as crucial as in the exchangeable case. Nevertheless, the probability distribution of $\mu$ under $P$ is meaningful. In fact, $\mu(A)$ is the long run frequency of the events $\{X_n\in A\}$. Similarly, because of \eqref{d5tni9}, $\mu(A)$ can be regarded as the asymptotically optimal predictor of the event $\{$the next observation belongs to $A\}$. Moreover, as noted above, the restriction of $P$ on $\sigma(\mu)$ is exchangeable.

Finally, we characterize c.i.d. sequences in terms of strategies. The next result is fundamental for this paper.

\begin{thm}{\bf (\cite[Th. 3.1]{BPR2012}).}\label{c5ty6} For any strategy $\sigma$, $(X_n)$ is c.i.d. under $P_\sigma$ if and only if
\begin{gather*}
\sigma_n(x)(A)=\int\sigma_{n+1}(x,y)(A)\,\sigma_n(x)(dy)
\end{gather*}
for all $n\ge 0$, all $A\in\mathcal{B}$ and $P_\sigma$-almost all $x\in S^n$.
\end{thm}

Henceforth, we just say ``$P_\sigma$ is c.i.d." to mean that ``$(X_n)$ is c.i.d. under $P_\sigma$".

\vspace{0.2cm}

\section{Convex combinations of random probability measures}\label{d4j8r}

Let $\nu\in\mathcal{P}$. Moreover, for each $n\ge 0$, let $q_n:S^n\rightarrow [0,1]$ be a measurable function (with $q_0$ constant) and $\alpha_n$ a kernel on $(S,\mathcal{B})$.

In this section, the strategy $\sigma$ satisfies equation \eqref{f5ty7}, namely
\begin{gather*}
\sigma_0=\nu\quad\text{and}\quad\sigma_{n+1}(x,y)=q_n(x)\,\sigma_n(x)+(1-q_n(x))\,\alpha_n(y)
\end{gather*}
for all $n\ge 0$, $x\in S^n$ and $y\in S$. Such a $\sigma$ has been investigated in \cite{BDPR2021} in the special case $\alpha_n=\alpha_0$ for all $n$. Here, the results of \cite{BDPR2021} are extended and improved. Some new examples are also obtained.

We first note that, arguing by induction, $\sigma$ can be written as

\begin{gather}\label{z5rt6}
\sigma_n(x)=\nu\,\prod_{j=0}^{n-1}q_j\,+\,\sum_{i=1}^n\alpha_{i-1}(x_i)\,(1-q_{i-1})\prod_{j=i}^{n-1}q_j
\end{gather}
for all $n\ge 1$ and $x=(x_1,\ldots,x_n)\in S^n$. In formula \eqref{z5rt6}, $\prod_{j=i}^{n-1}q_j$ is meant as 1 when $i=n$, and $q_j$ is a shorthand notation to denote
\begin{gather*}
q_j=q_j(x_1,\ldots,x_j).
\end{gather*}

We next give conditions for $P_\sigma$ to be c.i.d.

\begin{thm}\label{m9i8}
$P_\sigma$ is c.i.d. provided
\begin{gather*}
\nu(A)=\int\alpha_0(z)(A)\,\nu(dz)\quad\quad\text{and}
\\\alpha_n(y)(A)=\int\alpha_{n+1}(z)(A)\,\alpha_n(y)(dz)
\end{gather*}
for all $n\ge 0$, all $A\in\mathcal{B}$ and $\nu$-almost all $y\in S$. (Recall that $\sigma_0=\nu$).
\end{thm}

\begin{proof}
By Theorem \ref{c5ty6}, it suffices to find a set $C\in\mathcal{B}$ such that
\begin{gather*}
\nu(C)=1\quad\text{and}\quad\sigma_{n}(x)(A)=\int\sigma_{n+1}(x,y)(A)\,\sigma_n(x)(dy)
\end{gather*}
for all $n\ge 0$, $A\in\mathcal{B}$ and $x\in C^n$.

For each $n\ge 0$, if $\nu(\cdot)=\int\alpha_n(z)(\cdot)\,\nu(dz)$, then
\begin{gather*}
\int\alpha_{n+1}(z)(A)\,\nu(dz)=\int\int\alpha_{n+1}(z)(A)\,\alpha_n(y)(dz)\,\nu(dy)
\\=\int\alpha_n(y)(A)\,\nu(dy)=\nu(A)\quad\text{for each }A\in\mathcal{B}.
\end{gather*}
Hence, by induction, $\nu(\cdot)=\int\alpha_n(z)(\cdot)\,\nu(dz)$ for all $n\ge 0$.

By standard arguments, there is a set $F\in\mathcal{B}$ such that $\nu(F)=1$ and
\begin{gather*}
\alpha_{n}(y)(A)=\int\alpha_{n+1}(z)(A)\,\alpha_n(y)(dz)\quad\text{for all }n\ge 0,\,A\in\mathcal{B}\text{ and }y\in F.
\end{gather*}
Define $C_0=F$ and, for each $n\ge 0$,
\begin{gather*}
C_{n+1}=\bigl\{y\in C_n:\alpha_j(y)(C_n)=1\text{ for all }j\ge 0\bigr\}.
\end{gather*}

Since $\nu(C_n)=1$ implies $\nu(C_{n+1})=1$, one obtains $\nu(C_n)=1$ for all $n$. Hence, letting
\begin{gather*}
C=\cap_nC_n,
\end{gather*}
it follows that
\begin{gather*}
C\in\mathcal{B},\quad\nu(C)=1,\quad\alpha_n(y)(C)=1\text{ for all }n\ge 0\text{ and }y\in C,
\\\alpha_{n}(y)(A)=\int\alpha_{n+1}(z)(A)\,\alpha_n(y)(dz)\quad\text{for all }n\ge 0,\,A\in\mathcal{B}\text{ and }y\in C.
\end{gather*}

Finally, arguing still by induction (on $n-j$) it follows that
\begin{gather*}
\int\alpha_n(z)(A)\,\alpha_j(y)(dz)=\alpha_j(y)(A)\quad\quad\text{for all }0\le j<n,\,A\in\mathcal{B}\text{ and }y\in C.
\end{gather*}
Hence, because of \eqref{z5rt6}, one obtains
\begin{gather*}
\int\sigma_{n+1}(x,y)(A)\,\sigma_n(x)(dy)=q_n(x)\,\sigma_n(x)(A)+(1-q_n(x))\,\int\alpha_n(y)(A)\,\sigma_n(x)(dy)
\\=q_n(x)\,\sigma_n(x)(A)+(1-q_n(x))\,\sigma_n(x)(A)=\sigma_n(x)(A)
\end{gather*}
for all $n\ge 0$, $A\in\mathcal{B}$ and $x\in C^n$. This concludes the proof.
\end{proof}

The conditions of Theorem \ref{m9i8} are automatically true if each $\alpha_n$ is a version of the conditional probability of $\nu$ given $\mathcal{G}_n$, for some filtration $(\mathcal{G}_n)$ on $(S,\mathcal{B})$.

\begin{cor}\label{y8j}
Let $\mathcal{G}_0\subset\mathcal{G}_1\subset\mathcal{G}_2\subset\ldots\subset\mathcal{B}$ be an increasing sequence of sub-$\sigma$-fields of $\mathcal{B}$. Then, $P_\sigma$ is c.i.d. whenever
\begin{gather*}
\alpha_n(\cdot)(A)=E_\nu\bigl(1_A\mid\mathcal{G}_n\bigr),\quad\quad\nu\text{-a.s., for all }n\ge 0\text{ and }A\in\mathcal{B}.
\end{gather*}
\end{cor}

\begin{proof}
Just note that
\begin{gather*}
\nu(A)=E_\nu\Bigl\{E_\nu\bigl(1_A\mid\mathcal{G}_0\bigr)\Bigr\}=\int\alpha_0(z)(A)\,\nu(dz)\quad\quad\text{and}
\\\alpha_n(\cdot)(A)=E_\nu\bigl(1_A\mid\mathcal{G}_n\bigr)=E_\nu\Bigl\{E_\nu\bigl(1_A\mid\mathcal{G}_{n+1}\bigr)\mid\mathcal{G}_n\Bigr\}
\\=\int\alpha_{n+1}(z)(A)\,\alpha_n(\cdot)(dz),\quad\quad\nu\text{-a.s.}
\end{gather*}
\end{proof}

We are now able to provide examples of strategies which satisfy equation \eqref{f5ty7} and make $(X_n)$ c.i.d.

\begin{ex}\label{s5rt}\textbf{(Example 13 of \cite{BDPR2021}).}
For each $n\ge 0$, fix a countable partition $\mathcal{H}_n$ of $S$ such that $H\in\mathcal{B}$ and $\nu(H)>0$ for all $H\in\mathcal{H}_n$. Suppose $\mathcal{H}_{n+1}$ is finer than $\mathcal{H}_n$ and define
\begin{gather*}
\alpha_n(y)=\sum_{H\in\mathcal{H}_n}1_H(y)\,\nu(\cdot\mid H)=\nu\bigl[\cdot\mid H_n(y)\bigr]
\end{gather*}
where $H_n(y)$ denotes the only $H\in\mathcal{H}_n$ such that $y\in H$. Letting
\begin{gather*}
\mathcal{G}_n=\sigma(\mathcal{H}_n),
\end{gather*}
one obtains $\mathcal{G}_n\subset\mathcal{G}_{n+1}$ (since $\mathcal{H}_{n+1}$ is finer than $\mathcal{H}_n$) and $\alpha_n(\cdot)(A)=E_\nu\bigl(1_A\mid\mathcal{G}_n\bigr)$ for each $n\ge 0$ and $A\in\mathcal{B}$. Hence, $P_\sigma$ is c.i.d. because of Corollary \ref{y8j}.
\end{ex}

\vspace{0.2cm}

Example \ref{s5rt} can be developed in various ways. For any partition $\mathcal{H}$ of $S$, let
\begin{gather*}
\mathcal{U}(\mathcal{H})=\sup_{H\in\mathcal{H}}\,\,\sup_{y,z\in H}\,d(y,z)
\end{gather*}
where $d$ denotes the distance on $S$.

\vspace{0.2cm}

\begin{ex}\label{m9e}\textbf{(Dirichlet-like sequences).} Fix a constant $c>0$ and define
\begin{gather*}
q_n=\frac{n+c}{n+1+c}\quad\text{and}\quad\alpha_n(y)=\nu\bigl[\cdot\mid H_n(y)\bigr].
\end{gather*}
Then, formula \eqref{z5rt6} yields
\begin{gather*}
\sigma_n(x)=\frac{c\,\nu+\sum_{i=1}^n\nu\bigl[\cdot\mid H_{i-1}(x_i)\bigr]}{n+c}=\frac{c}{n+c}\,\nu\,+\,\frac{n}{n+c}\,\nu_n(x)
\end{gather*}
where
\begin{gather*}
\nu_n(x)=\frac{\sum_{i=1}^n\nu\bigl[\cdot\mid H_{i-1}(x_i)\bigr]}{n}.
\end{gather*}
In turn, the predictives of a Dirichlet sequence are
\begin{gather*}
\beta_n(x)=\frac{c}{n+c}\,\nu\,+\,\frac{n}{n+c}\,\mu_n(x)
\end{gather*}
where $\mu_n(x)=(1/n)\,\sum_{i=1}^n\delta_{x_i}$ is the empirical measure.

The strategies $\sigma$ and $\beta$ look alike, and $\sigma$ reduces to $\beta$ if $\nu\bigl[\cdot\mid H_{i-1}(x_i)\bigr]$ is replaced by $\delta_{x_i}$. Moreover, $\sigma_n(x)$ and $\beta_n(x)$ are usually close for large $n$. In fact, for various distances $D$ on $\mathcal{P}$, one obtains
\begin{gather}\label{c5t}
\lim_nD\bigl[\sigma_n(x),\,\beta_n(x)\bigr]=0\quad\quad\text{for each }x\in S^\infty
\end{gather}
provided $\lim_n\mathcal{U}(\mathcal{H}_n)=0$. For instance, relation \eqref{c5t} holds if $D$ is the bounded Lipschitz metric; see Theorem \ref{n8y7u}.

Despite \eqref{c5t}, however, $\sigma$ and $\beta$ conflict under a fundamental aspect. Indeed,
\begin{gather*}
\sigma_n(x)\ll\nu\quad\quad\text{for all }n\ge 0\text{ and }x\in S^n
\end{gather*}
while this is not true for $\beta_n(x)$. For instance, if $\nu$ is diffuse and $n$ is large, $\beta_n(x)$ is very close to be singular with respect to $\nu$. This striking difference implies that $P_\sigma$ and $P_\beta$ are singular when $\nu$ is diffuse; see Theorem \ref{n8y7u} again.

\end{ex}

\vspace{0.2cm}

\begin{ex}\label{d5n9j}\textbf{(Example \ref{m9e} continued).} The situation of Example \ref{m9e} may appear strange. Suppose $\nu$ is diffuse. On one hand, since $P_\sigma$ and $P_\beta$ are singular, $\sigma$ and $\beta$ induce completely different distributions on the data sequence $(X_n)$. On the other hand, because of \eqref{c5t}, $\sigma$ and $\beta$ provide similar predictions for large $n$.

Such a situation mostly depends on the distance $D$. In fact, $\sigma_n(x)$ and $\beta_n(x)$ are no longer close if $D$ is replaced by some stronger distance on $\mathcal{P}$, such as the total variation distance.

More precisely, fix a bounded measurable function $f:S\rightarrow\mathbb{R}$ and suppose the target is to predict $f(X_{n+1})$ based on $(X_1,\ldots,X_n)$. Then, $\sigma$ and $\beta$ actually yield similar predictions for large $n$. As an example, if $f$ is Lipschitz and $D$ is the bounded Lipschitz metric, one obtains
\begin{gather*}
\Abs{E_\sigma\Bigl\{f(X_{n+1})\mid (X_1,\ldots,X_n)=x\Bigr\}-E_\beta\Bigl\{f(X_{n+1})\mid (X_1,\ldots,X_n)=x\Bigr\}}
\\=\Abs{\int f(t)\,\sigma_n(x)(dt)-\int f(t)\,\beta_n(x)(dt)}\le k\,D\bigl[\sigma_n(x),\,\beta_n(x)\bigr]
\end{gather*}
for some constant $k$ depending only on $f$.

However, $\sigma$ and $\beta$ give conflicting predictions in more elaborated problems. For instance, suppose one aims to predict whether or not the next observation is new. Letting $G_n=\{X_{n+1}=X_i\text{ for some }i\le n\}$, for every $x\in S^n$ one obtains
\begin{gather*}
P_\sigma\bigl(G_n\mid (X_1,\ldots,X_n)=x\bigr)=\sigma_n(x)(\{x_1,\ldots,x_n\})=0
\end{gather*}
while
\begin{gather*}
P_\beta\bigl(G_n\mid (X_1,\ldots,X_n)=x\bigr)=\beta_n(x)(\{x_1,\ldots,x_n\})=n/(n+c).
\end{gather*}
\end{ex}

\vspace{0.2cm}

\begin{ex}\label{t5xd4}\textbf{(Exponential smoothing-like sequences).} Let
\begin{gather*}
\beta_n(x)=q^n\nu+(1-q)\sum_{i=1}^nq^{n-i}\delta_{x_i}
\end{gather*}
where $q\in [0,1]$ is any constant. Making predictions through $\beta$ may be reasonable when the forecaster has only vague opinions on the dependence structure of the data, and yet she feels that the weight of the $i$-th observation $x_i$ should be an increasing function of $i$; see \cite{BCL} and \cite{BDPR2021}. Now, if $q_n=q$ and $\alpha_n(y)=\nu\bigl[\cdot\mid H_n(y)\bigr]$, formula \eqref{z5rt6} reduces to
\begin{gather*}
\sigma_n(x)=q^n\nu+(1-q)\sum_{i=1}^nq^{n-i}\nu\bigl[\cdot\mid H_{i-1}(x_i)\bigr].
\end{gather*}
Essentially the same remarks of Examples \ref{m9e}-\ref{d5n9j}, about the connections between $\sigma$ and $\beta$, can be repeated in this example.
\end{ex}

\vspace{0.2cm}

The next example deals with a more elaborate choice of $q_n$.

\vspace{0.2cm}

\begin{ex}\label{ali9}\textbf{(Reinforcements).} For each $n\ge 1$, fix a set $C_n\in\mathcal{B}^n$, two constants $0<a_n<1/2<b_n<1$, and define
\begin{gather*}
q_n(x)=b_n\,1_{C_n}(x)+a_n\,(1-1_{C_n}(x))\quad\quad\text{for all }x\in S^n.
\end{gather*}
Roughly speaking, the underlying idea is that $\sigma_n(x)$ exhibits good predictive performances whenever $x\in C_n$. Hence, if $x\in C_n$, to predict $x_{n+2}$ based on $(x,x_{n+1})$, the forecaster is inclined to reinforce $\sigma_n(x)$ with respect to $\alpha_n(x_{n+1})$. (Recall that $a_n<1/2<b_n$).

As a concrete example, take $S=[0,1]$ and $\sigma$ as in Example \ref{s5rt}. Moreover, let $\overline{x}_n=(1/n)\,\sum_{i=1}^nx_i$ be the sample mean of $x\in S^n$ and $m_n(x)$ any (measurable) predictor of $x_{n+1}$ based on $\sigma_n(x)$. For definiteness,
\begin{gather*}
m_n(x)=\int t\,\sigma_n(x)(dt).
\end{gather*}
If $m_n(x)$ is regarded as a predictor of the past observations $x_i$, $i\le n$, then
\begin{gather*}
\overline{x}_n-m_n(x)=(1/n)\,\sum_{i=1}^n\bigl\{x_i-m_n(x)\bigr\}
\end{gather*}
is the arithmetic mean of the prediction errors. In a sense, $\sigma_n(x)$ works nicely whenever $\overline{x}_n-m_n(x)$ is small. Therefore, given $\epsilon>0$, one could define
\begin{gather*}
C_n=\bigl\{x\in S^n:\abs{\overline{x}_n-m_n(x)}<\epsilon\bigr\}.
\end{gather*}

Two remarks are in order. First, since $P_\sigma$ is c.i.d.,
\begin{gather*}
\overline{x}_n-m_n(x)\longrightarrow 0\quad\quad\text{for }P_\sigma\text{-almost all }x\in S^\infty.
\end{gather*}
Hence, a.s., the events $C_n$ are eventually true. Second, the previous naive idea could be realized with other choices of $C_n$. For instance,
\begin{gather*}
C_n=\Bigl\{x\in S^n:\Abs{(1/n)\,\sum_{i=1}^n\bigl(x_i-m_{i-1}(x_1,\ldots,x_{i-1})\bigr)}<\epsilon\Bigr\}\quad\quad\text{or}
\\C_n=\Bigl\{x\in S^n:\sup_{0\le t\le 1}\Abs{\mu_n(x)([0,t])-\sigma_n(x)([0,t])}<\epsilon\Bigr\}
\end{gather*}
where $\mu_n(x)=(1/n)\,\sum_{i=1}^n\delta_{x_i}$ is the empirical measure; see e.g. \cite{BRR1997} and \cite{BCPR2009}.
\end{ex}

\vspace{0.2cm}

In the last example, we focus on the special case $\alpha_n=\alpha$ for each $n\ge 0$.

\vspace{0.2cm}

\begin{ex} \textbf{(Ad hoc choice of $\nu$).}
Let $\alpha$ be a kernel on $(S,\mathcal{B})$. If $\alpha_n=\alpha$ for all $n\ge 0$, to apply Corollary \ref{y8j}, it suffices to find a sub-$\sigma$-field $\mathcal{G}\subset\mathcal{B}$ such that
\begin{gather}\label{A3Q11}
\alpha(\cdot)(A)=E_\nu\bigl(1_A\mid\mathcal{G}\bigr),\quad\quad\nu\text{-a.s., for all }A\in\mathcal{B}.
\end{gather}
Usually, $\nu$ is given and one looks for $\alpha$ satisfying condition \eqref{A3Q11}. But the opposite route is admissible as well. Accordingly, in this example, we fix a (suitable) kernel $\alpha$ and we build $\nu$ so as to make equation \eqref{A3Q11} true.

For each kernel $\alpha$ on $(S,\mathcal{B})$, let
\begin{gather*}
C=\bigl\{y\in S:\alpha(y)=\delta_y\text{ on }\sigma(\alpha)\bigr\}
\end{gather*}
where $\sigma(\alpha)$ denotes the $\sigma$-field over $S$ generated by the maps $y\mapsto\alpha(y)(A)$ for all $A\in\mathcal{B}$. Now, fix a kernel $\alpha$ such that $C\ne\emptyset$, a probability $\nu_0\in\mathcal{P}$ supported by $C$, and define
\begin{gather*}
\nu(A)=\int\alpha(y)(A)\,\nu_0(dy)\quad\quad\text{for all }A\in\mathcal{B}.
\end{gather*}
Then,
\begin{gather*}
\alpha(y)(A\cap B)=\alpha(y)(A)\,1_B(y)\quad\quad\text{whenever }A\in\mathcal{B},\,B\in\sigma(\alpha)\text{ and }y\in C.
\end{gather*}
Hence, $\sigma_0(C)=1$ implies
\begin{gather*}
\nu(A\cap B)=\int\alpha(y)(A\cap B)\,\nu_0(dy)=\int\alpha(y)(A)\,1_B(y)\,\nu_0(dy)
\end{gather*}
for all $A\in\mathcal{B}$ and $B\in\sigma(\alpha)$. Letting $A=S$, one obtains $\nu=\nu_0$ on $\sigma(\alpha)$, and the above equation can be rewritten as
\begin{gather*}
\nu(A\cap B)=\int\alpha(y)(A)\,1_B(y)\,\nu(dy).
\end{gather*}
Therefore, equation \eqref{A3Q11} holds with $\mathcal{G}=\sigma(\alpha)$.

For instance, take a kernel $\alpha$, a measurable function $f:S\rightarrow\mathbb{R}$, and suppose that
\begin{gather*}
\alpha(y)=\alpha(y_0)\text{ if }f(y)=f(y_0)\quad\text{and}\quad\alpha(y_0)\Bigl(\{y:f(y)=f(y_0)\}\Bigr)=1
\end{gather*}
for some point $y_0\in S$. Then, $y_0\in C$ and  any $\nu_0$ supported by $C$ could be chosen.
\end{ex}

\vspace{0.2cm}

To conclude this section, we highlight that the family of predictive
distributions introduced in equation \eqref{z5rt6} has applications beyond the
predictive inferential framework of this paper. For instance, a well-known collection of species sampling sequences, namely
the Dirichlet sequences, is recovered by Example \ref{m9e}. Similarly, the Beta-GOS processes of \cite{ACBLG2014} are actually special cases of equation \eqref{z5rt6}. Accidentally, this has an impact in Bayesian nonparametrics where species sampling sequences are used to define priors.

\vspace{0.2cm}

\section{Predictions via stable laws}\label{n8d3wa}

In this section, we let $S=\mathbb{R}$, we fix a constant $\gamma\in (0,2]$, and we introduce a certain class of strategies. Each element $\sigma$ of such a class satisfies conditions (i)-(ii) and the probability measure $\sigma_n(x)$ is $\gamma$-stable for all $n\ge 0$ and $x\in S^n$. (The exponent $\gamma$ of a stable law is usually denoted by $\alpha$, but in this paper $\alpha$ is used to denote kernels).

Let $Z$ be a real random variable with characteristic function
\begin{gather*}
E\bigl\{\exp(i\,t\,Z)\bigr\}=\exp\Bigl(-\frac{\abs{t}^\gamma}{2}\Bigr)\quad\quad\text{for all }t\in\mathbb{R}.
\end{gather*}
For $a\in\mathbb{R}$ and $b>0$, denote by $\mathcal{S}(a,b)$ the probability distribution of $a+b^{1/\gamma}Z$, namely
\begin{gather*}
\mathcal{S}(a,b)(A)=P\bigl(a+b^{1/\gamma}Z\in A)\quad\quad\text{for all }A\in\mathcal{B}.
\end{gather*}
Next, fix the real numbers
\begin{gather*}
0=u_0<u_1<u_2<\ldots< u,
\end{gather*}
and define $f_0=0$ and
\begin{gather*}
f_{n+1}(x,y)=f_n(x)\,\left(1-\left(\frac{u_{n+1}-u_n}{u-u_n}\right)^{1/\gamma}\right)\,+\,y\,\,\left(\frac{u_{n+1}-u_n}{u-u_n}\right)^{1/\gamma}
\end{gather*}
for all $n\ge 0$, $x\in S^n$ and $y\in S$.

In this section, we focus on the strategy
\begin{gather}\label{r6g7}
\sigma_n(x)=\mathcal{S}\Bigl(f_n(x),\,u-u_n\Bigr)\quad\quad\text{for all }n\ge 0\text{ and }x\in S^n.
\end{gather}
It is worth noting that $\sigma_0=\mathcal{S}(0,u)$ and $\sigma_{n+1}(x,y)$ can be easily evaluated based on $y$ and the median of $\sigma_n(x)$. Hence, condition (ii) holds. We now prove condition (i).

\begin{thm}\label{t6m9u}
If $\sigma$ is given by \eqref{r6g7}, then $P_\sigma$ is c.i.d.
\end{thm}

\begin{proof}
By Theorem \ref{c5ty6}, it suffices to show that
\begin{gather*}
\sigma_n(x)(A)=\int\sigma_{n+1}(x,y)(A)\,\sigma_n(x)(dy)
\end{gather*}
for all $n\ge 0$, $A\in\mathcal{B}$ and $x\in S^n$. We need the following claim.

\vspace{0.2cm}

\textbf{Claim:} Let $a,\,v\in\mathbb{R}$ and $b,\,c>0$. If $Y\sim\mathcal{S}(a,b)$, then
\begin{gather*}
v+c^{1/\gamma}Y\sim\mathcal{S}\Bigl(v+ac^{1/\gamma},\,bc\Bigr)\quad\quad\text{and}
\\E\Bigl\{\mathcal{S}(Y,\,c)(A)\Bigr\}=\mathcal{S}(a,\,b+c)(A)\quad\quad\text{for all }A\in\mathcal{B}.
\end{gather*}

\vspace{0.2cm}

\textbf{Proof of the Claim:} Since $Y\sim a+b^{1/\gamma}Z$,
\begin{gather*}
v+c^{1/\gamma}Y\sim v+c^{1/\gamma}\bigl(a+b^{1/\gamma}Z\bigr)=v+ac^{1/\gamma}+(bc)^{1/\gamma}Z\sim\mathcal{S}\Bigl(v+ac^{1/\gamma},\,bc\Bigr).
\end{gather*}
To prove the second part, take a random variable $T$ independent of $Y$ such that $T\sim\mathcal{S}(0,c)$. Then, $T+Y\sim\mathcal{S}(a,\,b+c)$ and this implies
\begin{gather*}
\mathcal{S}(a,\,b+c)(A)=P(T+Y\in A)=\int P(T+y\in A)\,\mathcal{S}(a,\,b)(dy)
\\=\int\mathcal{S}(y,\,c)(A)\,\mathcal{S}(a,\,b)(dy)=E\Bigl\{\mathcal{S}(Y,\,c)(A)\Bigr\}.
\end{gather*}

\vspace{0.2cm}

We now come back to the Theorem. Fix $n\ge 0$, $x\in S^n$, and define
\begin{gather*}
a=f_n(x),\quad b=u-u_n,\quad v=f_n(x)\,\left(1-\left(\frac{u_{n+1}-u_n}{u-u_n}\right)^{1/\gamma}\,\right),\quad c=\frac{u_{n+1}-u_n}{u-u_n}.
\end{gather*}
Then, $\sigma_n(x)=\mathcal{S}\Bigl(f_n(x),\,u-u_n\Bigr)=\mathcal{S}(a,b)$ and $f_{n+1}(x,y)=v+c^{1/\gamma}y$. By the Claim, if $Y\sim\sigma_n(x)$, then
\begin{gather*}
Y^*:=f_{n+1}(x,Y)\sim\mathcal{S}\Bigl(f_n(x),\,u_{n+1}-u_n\Bigr).
\end{gather*}
Therefore, applying the Claim with $a=f_n(x)$ and $b=u_{n+1}-u_n$, one obtains
\begin{gather*}
\int\sigma_{n+1}(x,y)(A)\,\sigma_n(x)(dy)=\int\mathcal{S}\Bigl(f_{n+1}(x,y),\,u-u_{n+1}\Bigr)(A)\,\sigma_n(x)(dy)
\\=E\Bigl\{\mathcal{S}\bigl(Y^*,\,u-u_{n+1}\bigr)(A)\Bigr\}=\mathcal{S}\Bigl(f_n(x),\,u-u_n\Bigr)(A)=\sigma_n(x)(A).
\end{gather*}
This concludes the proof.
\end{proof}

\vspace{0.2cm}

In the rest of this section, $\sigma$ always denotes the strategy \eqref{r6g7}.

An useful feature of $\sigma$ is its asymptotic behavior, which can be determined quite easily. Define in fact
\begin{gather*}
L=\bigl\{x\in S^\infty:\lim_nf_n(x)\text{ exists and is finite}\bigr\}
\end{gather*}
and $f(x)=\lim_nf_n(x)$ for each $x\in L$. Since $P_\sigma$ is c.i.d., it follows that $P_\sigma(L)=1$. And, for each $x\in L$, one obtains
\begin{gather*}
\sigma_n(x)\longrightarrow\delta_{f(x)}\,\text{ weakly if }\,\sup_nu_n=u\,\text{ and}
\\\sigma_n(x)\longrightarrow\mathcal{S}\Bigl(f(x),\,u-\sup_nu_n\Bigr)\,\text{ in total variation if }\,\sup_nu_n<u.
\end{gather*}
We refer to the proof of Theorem \ref{z5t6yh8} for more details. Here, we turn to examples.

\begin{ex}\textbf{(Cauchy and Normal distributions).}
The most popular cases are $\gamma=1$ and $\gamma=2$. Let $\mathcal{C}(a,b)$ denote the probability measure
\begin{gather*}
\mathcal{C}(a,b)(A)=\frac{2\,b}{\pi}\,\int_A\,\frac{1}{b^2+4\,(t-a)^2}\,dt\quad\quad\text{for all }A\in\mathcal{B}.
\end{gather*}
(Note that, in this parametrization, the standard Cauchy distribution is $\mathcal{C}(0,2)$ and not $\mathcal{C}(0,1)$). Then,
\begin{gather*}
\sigma_n(x)=\mathcal{C}\Bigl(f_n(x),\,u-u_n\Bigr)\quad\text{or}\quad\sigma_n(x)=\mathcal{N}\Bigl(f_n(x),\,u-u_n\Bigr)
\end{gather*}
according to whether $\gamma=1$ or $\gamma=2$. Both strategies can be useful in real problems. Note also that $f_n(x)$ is just a weighted average of the first $n$ observations $x_1,\ldots,x_n$ and, in the normal case, the weights are connected to the conditional variances.
\end{ex}

\vspace{0.2cm}

The next example provides further information on the data sequence $(X_n)$.

\vspace{0.2cm}

\begin{ex}\textbf{(Finite dimensional distributions).}\label{h7n3w} Let
\begin{gather*}
Y_{n+1}=\sum_{i=1}^n(u_i-u_{i-1})^{1/\gamma}\,Z_i+(u-u_n)^{1/\gamma}\,Z_{n+1}\quad\text{for all }n\ge 0,
\end{gather*}
where $Z_1,Z_2,\ldots$ is an i.i.d. sequence with $Z_1\sim\mathcal{S}(0,1)$. Then, $Y_1\sim\mathcal{S}(0,u)$. Furthermore,
\begin{gather*}
(Y_1,\ldots,Y_n)=g_n(Z_1,\ldots,Z_n)\quad\text{and}\quad\sum_{i=1}^n(u_i-u_{i-1})^{1/\gamma}\,Z_i=f_n(Y_1,\ldots,Y_n)
\end{gather*}
where $g_n$ is an invertible linear transformation. Therefore,
\begin{gather*}
P\bigl(Y_{n+1}\in\cdot\mid Y_1,\ldots,Y_n\bigr)=P\bigl(Y_{n+1}\in\cdot\mid Z_1,\ldots,Z_n\bigr)
\\=P\left(\,f_n(Y_1,\ldots,Y_n)+(u-u_n)^{1/\gamma}\,Z_{n+1}\in\cdot\mid Z_1,\ldots,Z_n\right)
\\=\mathcal{S}\Bigl(f_n(Y_1,\ldots,Y_n),\,u-u_n\Bigr)=\sigma_n(Y_1,\ldots,Y_n)\quad\quad\text{a.s.}
\end{gather*}
In other terms, the predictive distributions of the sequence $(Y_n)$ agree with those of $\sigma$, and this implies
\begin{gather*}
P_\sigma(B)=P\bigl((Y_1,Y_2,\ldots)\in B\bigr)\quad\quad\text{for all }B\in\mathcal{B}^\infty.
\end{gather*}
This equation allows to determine the finite dimensional distributions of $(X_n)$ under $P_\sigma$. Here, we just highlight two facts. Firstly,
\begin{gather*}
f_n(Y_1,\ldots,Y_n)=\sum_{i=1}^n(u_i-u_{i-1})^{1/\gamma}\,Z_i\sim u_n^{1/\gamma}\,Z_1\sim\mathcal{S}(0,u_n).
\end{gather*}
Thus, $f_n\sim\mathcal{S}(0,u_n)$ under $P_\sigma$, namely, $P_\sigma(f_n\in A)=\mathcal{S}(0,u_n)(A)$ for all $A\in\mathcal{B}$. Secondly, since $g_n$ is linear, the finite dimensional distributions of $(X_n)$ under $P_\sigma$ are Gaussian when $\gamma=2$. In this case, since $(Y_n)$ is c.i.d., the moments are
\begin{gather*}
E_{P_\sigma}(X_n)=0,\quad E_{P_\sigma}(X_n^2)=u\quad\quad\text{and}
\\E_{P_\sigma}(X_nX_m)=E(Y_nY_m)=E\bigl[Y_n\,E(Y_m\mid Y_1,\ldots,Y_n)\bigr]
\\=E(Y_n\,Y_{n+1})=u_{n-1}+\sqrt{(u_n-u_{n-1})(u-u_{n-1})}\quad\quad\text{for all }1\le n<m.
\end{gather*}
\end{ex}

\vspace{0.2cm}

The last example collects some miscellaneous remarks.

\vspace{0.2cm}

\begin{ex}\label{sd5rf7l}\textbf{(Choice of $\gamma$, $u$ and $u_n$).} To work with $\sigma$, one has only to select $\gamma$ and $u,u_1,u_2,\ldots$ Obviously, the choice of $\gamma$ depends on the specific problem at hand. We just note that, in applications, $\gamma\in\{1,2\}$ is not the unique meaningful choice. For instance, $\gamma\notin\{1,2\}$ is quite common when modeling financial data; see e.g. \cite[Chap. 13]{MCUL}. The numbers $u$ and $u_n$ are scale parameters which control the dispersion structure of $(X_n)$. If $\gamma=2$, for instance, $u$ and $u_n$ determine the variances and covariances of the Gaussian sequence $(X_n)$; see Example \ref{h7n3w}. An important distinguish is $\sup_nu_n=u$ or $\sup_nu_n<u$, as the limiting distribution of $\sigma_n$ is degenerate in the former case while it is not in the latter. Finally, we mention a practically useful choice of $u_n$. Fix $u>0$ and $q\in (0,1)$ and define
\begin{gather*}
u_n=u\,(1-q^n)\quad\quad\text{for all }n\ge 0.
\end{gather*}
Then, $u_{n+1}-u_n=(u-u_n)(1-q)$ and the updating rule for $f_n$ reduces to
\begin{gather*}
f_{n+1}(x,y)=(1-b)\,f_n(x)+b\,y\quad\quad\text{where }b=(1-q)^{1/\gamma}.
\end{gather*}
Equivalently, $f_n(x)=b\,\sum_{j=1}^n(1-b)^{n-j}x_j$ for each $x\in S^n$.
\end{ex}

\vspace{0.2cm}

\section{Asymptotics}\label{s4r6}

We first recall two popular distances on $\mathcal{P}$. Let $\lambda_1,\,\lambda_2\in\mathcal{P}$ and let $F$ be the set of all functions $f:S\rightarrow [-1,1]$ such that $\abs{f(y)-f(z)}\le d(y,z)$ for all $y,\,z\in S$, where $d$ is the distance on $S$. The {\em bounded Lipschitz metric} and the {\em total variation} distance are, respectively,
\begin{gather*}
D(\lambda_1,\lambda_2)=\sup_{f\in F}\,\,\Abs{\int f\,d\lambda_1-\int f\,d\lambda_2}\quad\text{and}\quad\norm{\lambda_1-\lambda_2}=\sup_{A\in\mathcal{B}}\,\abs{\lambda_1(A)-\lambda_2(A)}.
\end{gather*}
It is not hard to see that $D\le 2\,\norm{\cdot}$. Moreover, $D$ metrizes weak convergence of probability measures, in the sense that, for all $\lambda_n,\,\lambda\in\mathcal{P}$,
\begin{gather*}
\lambda_n\rightarrow\lambda\text{ weakly}\quad\Leftrightarrow\quad\lim_nD(\lambda_n,\lambda)=0.
\end{gather*}

We next prove some claims made in Example \ref{m9e}.

\begin{thm}\label{n8y7u}
Let $\sigma$ and $\beta$ be as in Example \ref{m9e}. If $\lim_n\mathcal{U}(\mathcal{H}_n)=0$, then
\begin{gather*}
\lim_nD\bigl[\sigma_n(x),\,\beta_n(x)\bigr]=0\quad\quad\text{for each }x\in S^\infty.
\end{gather*}
Moreover, $P_\sigma$ and $P_\beta$ are singular if $\nu$ is diffuse.
\end{thm}

\begin{proof}
Suppose $\lim_n\mathcal{U}(\mathcal{H}_n)=0$ and fix $x\in S^\infty$. It can be assumed $\mathcal{U}(\mathcal{H}_n)<\infty$ for all $n\ge 0$. Since $\abs{f(y)-f(z)}\le d(y,z)$ for all $f\in F$ and $y,\,z\in S$, one obtains
\begin{gather*}
D\bigl[\sigma_n(x),\,\beta_n(x)\bigr]=\sup_{f\in F}\,\,\Abs{\int f(t)\,\sigma_n(x)(dt)-\int f(t)\,\beta_n(x)(dt)}
\\=\frac{n}{n+c}\,\sup_{f\in F}\,\,\Abs{\int f(t)\,\nu_n(x)(dt)-\int f(t)\,\mu_n(x)(dt)}
\\=\frac{1}{n+c}\,\sup_{f\in F}\,\,\Abs{\,\sum_{i=1}^n\int\bigl\{f(t)-f(x_i)\bigr\}\,\nu\bigl[dt\mid H_{i-1}(x_i)\bigr]}
\\\le\frac{1}{n+c}\,\sum_{i=1}^n\int d(t,x_i)\,\nu\bigl[dt\mid H_{i-1}(x_i)\bigr]
\\\le\frac{1}{n+c}\,\sum_{i=1}^n\mathcal{U}(\mathcal{H}_{i-1})\longrightarrow 0\quad\quad\text{as }n\rightarrow\infty.
\end{gather*}

Finally, suppose $\nu$ is diffuse and define
\begin{gather*}
G_n=\bigl\{X_{n+1}=X_i\text{ for some }i\le n\bigr\}\quad\text{and}\quad G=\bigcup_n G_n.
\end{gather*}
As in Example \ref{d5n9j}, for all $x\in S^n$, one obtains
\begin{gather*}
P_\sigma\bigl(G_n\mid (X_1,\ldots,X_n)=x\bigr)=0\quad\text{and}\quad P_\beta\bigl(G_n\mid (X_1,\ldots,X_n)=x\bigr)=n/(n+c).
\end{gather*}
Therefore, $P_\sigma(G_n)=0$ and $P_\beta(G_n)=n/(n+c)$ for all $n$, which in turn implies
\begin{gather*}
P_\sigma(G)=P_\beta(G^c)=0.
\end{gather*}
\end{proof}

Next, for each $x\in S^\infty$, define
\begin{gather*}
\mu(x)=\lim_n\mu_n(x)\,\text{ if the limit exists and }\,\mu(x)=\delta_{x_1}\,\text{ otherwise,}
\end{gather*}
where $\mu_n(x)=(1/n)\,\sum_{i=1}^n\delta_{x_i}$ is the empirical measure and the limit is meant as a weak limit of probability measures. Then,
\begin{gather*}
P_\sigma\Bigl\{x\in S^\infty:\sigma_n(x)\rightarrow\mu(x)\text{ weakly}\Bigr\}=1
\end{gather*}
for any strategy $\sigma$ such that $P_\sigma$ is c.i.d.; see relation \eqref{d5tni9}.

The random probability measure $\mu$ is a meaningful object; see Subsection \ref{b7y}. In the sequel, we investigate $\mu$ when $\sigma$ comes from Sections \ref{d4j8r}-\ref{n8d3wa}.

For each $\tau\in\mathcal{P}$, say that $\tau$ is degenerate if $\tau=\delta_z$ for some $z\in S$. The abbreviation ``a.s." stands for ``$P_\sigma$-a.s." For instance, if $\nu\in\mathcal{P}$, we write $\mu\ll\nu$ a.s. to mean
\begin{gather*}
\mu(x)\ll\nu\quad\quad\text{for }P_\sigma\text{-almost all }x\in S^\infty.
\end{gather*}
Recall also that $q_n(x)=q_n(x_1,\ldots,x_n)$ for all $x\in S^\infty$.

\begin{thm}\label{u9mh1}
If the strategy $\sigma$ satisfies equation \eqref{f5ty7}, then $\sigma_n(x)$ converges in total variation distance for each $x\in S^\infty$ such that $\sum_n(1-q_n(x))<\infty$. Moreover, if $\sigma$ is as in Example \ref{s5rt}, then:

\vspace{0.2cm}

\begin{itemize}

\item $\mu\ll\nu$ a.s. and $\lim_n\norm{\sigma_n-\mu}=0$ a.s. provided $\sum_n(1-q_n)<\infty$ a.s.;

\vspace{0.2cm}

\item $\mu$ is degenerate a.s. provided $\lim_n\mathcal{U}(\mathcal{H}_n)=0$ and there are constants $a>0$ and $c_n\ge 0$ such that

\end{itemize}
\begin{gather}\label{z3e9ij}
\sum_nc_n^2=\infty\quad\text{and}\quad a\le q_n\le 1-c_n\text{ a.s. for all }n\ge 0.
\end{gather}
\end{thm}

\begin{proof}
Fix $x\in S^\infty$. By \eqref{z5rt6}, for all $n,\,k\ge 1$, one obtains
\begin{gather*}
\sigma_{n+k}(x)=\nu\,\prod_{j=0}^{n+k-1}q_j(x)\,+\,\sum_{i=1}^{n+k}\alpha_{i-1}(x_i)\,(1-q_{i-1}(x))\prod_{j=i}^{n+k-1}q_j(x)
\\=\sigma_n(x)\,\prod_{j=n}^{n+k-1}q_j(x)\,+\,\sum_{i=n+1}^{n+k}\alpha_{i-1}(x_i)\,(1-q_{i-1}(x))\prod_{j=i}^{n+k-1}q_j(x).
\end{gather*}
Therefore,
\begin{gather*}
\norm{\sigma_n(x)-\sigma_{n+k}(x)}\le 1-\prod_{j=n}^{n+k-1}q_j(x)\,+\,\sum_{i=n+1}^{n+k}(1-q_{i-1}(x)).
\end{gather*}
If $\sum_n(1-q_n(x))<\infty$, then $\prod_{j=n}^\infty q_j(x)$ is well defined and $\prod_{j=n}^\infty q_j(x)\le\prod_{j=n}^{n+k-1}q_j(x)$ for all $n$ and $k$. It follows that
\begin{gather*}
\sup_k\,\norm{\sigma_n(x)-\sigma_{n+k}(x)}\le 1-\prod_{j=n}^\infty q_j(x)\,+\,\sum_{i=n}^\infty(1-q_i(x))\longrightarrow 0\quad\text{as }n\rightarrow\infty.
\end{gather*}
Hence, $\sigma_n(x)$ converges in total variation distance since $(\mathcal{P},\,\norm{\cdot})$ is a complete metric space and $\sigma_n(x)$ is a Cauchy sequence.

\vspace{0.2cm}

Next, suppose $\sigma$ is as in Example \ref{s5rt} and $\sum_n(1-q_n)<\infty$ a.s. Since $P_\sigma$ is c.i.d., $\sigma_n\rightarrow\mu$ weakly a.s. Hence, the first part of this proof implies
\begin{gather*}
\lim_n\,\norm{\sigma_n-\mu}=0\quad\text{a.s.}
\end{gather*}
Furthermore, since $\sigma_n(x)\ll\nu$ for all $n\ge 0$ and $x\in S^n$, the joint distribution of $(X_1,\ldots,X_n)$ is absolutely continuous with respect to $\nu^n$ for all $n\ge 1$. Therefore, $\mu\ll\nu$ a.s. follows from \cite[Th. 1]{BPR2013}.

\vspace{0.2cm}

Finally, suppose $\sigma$ is as in Example \ref{s5rt}, $\lim_n\mathcal{U}(\mathcal{H}_n)=0$, and condition \eqref{z3e9ij} holds. To prove that $\mu$ is degenerate a.s., it suffices to show that, for each $f\in F$, there is a subsequence $(n_j)$ such that
\begin{gather*}
\int f(t)\,\mu(x)(dt)=\lim_jf(x_{n_j})\quad\quad\text{for }P_\sigma\text{-almost all }x\in S^\infty.
\end{gather*}
Since $\sigma_n\rightarrow\mu$ weakly a.s., this relation is equivalent to
\begin{gather*}
\lim_j\int \bigl\{f(t)-f(x_{n_j})\bigr\}\,\sigma_{n_j}(x)(dt)=0\quad\quad\text{for }P_\sigma\text{-almost all }x\in S^\infty.
\end{gather*}
We just give a sketch of the proof of the above limit relation.

\vspace{0.2cm}

Fix $f\in F$ and define
\begin{gather*}
\Delta_n(x)=\int f(t)\,\sigma_n(x)(dt)-\int f(t)\,\alpha_{n-1}(x_n)(dt)\quad\quad\text{for all }x\in S^\infty.
\end{gather*}
Using \eqref{z3e9ij} and arguing as in the proof of \cite[Th. 3]{BDPR2021BIS}, it can be shown that
\begin{gather*}
\liminf_n\int\Delta_n(x)^2\,P_\sigma(dx)=0.
\end{gather*}
Hence, there is a subsequence $(n_j)$ such that $\Delta_{n_j}\overset{a.s.}\longrightarrow 0$ as $j\rightarrow\infty$. Recalling that $\alpha_{n-1}(x_n)=\nu\bigl[\cdot\mid H_{n-1}(x_n)\bigr]$, one also obtains
\begin{gather*}
\Abs{\int f(t)\,\alpha_{n-1}(x_n)(dt)-f(x_n)}\le\int\abs{f(t)-f(x_n)}\,\nu\bigl[dt\mid H_{n-1}(x_n)\bigr]\le\mathcal{U}(\mathcal{H}_{n-1})\longrightarrow 0.
\end{gather*}
This concludes the proof.
\end{proof}

\vspace{0.2cm}

Theorem \ref{u9mh1} can be applied to the examples of Section \ref{d4j8r}. Suppose in fact $\lim_n\mathcal{U}(\mathcal{H}_n)=0$. Then, in Example \ref{t5xd4}, $\mu$ is degenerate a.s. In Example \ref{ali9}, $\mu\ll\nu$ a.s. if $\sum_n(1-b_n)<\infty$ and $\mu$ is degenerate a.s. if $\sum_n(1-b_n)^2=\infty$ and $\inf_na_n>0$. However, Theorem \ref{u9mh1} does not work in Example \ref{m9e}, for in that case
\begin{gather*}
\sum_n(1-q_n(x))=\infty\quad\text{and}\quad\sum_n(1-q_n(x))^2<\infty\quad\text{for all }x\in S^\infty.
\end{gather*}
Indeed, the behavior of $\mu$ in Example \ref{m9e} is an open problem.

Finally, we turn to the strategies of Section \ref{n8d3wa}.

\begin{thm}\label{z5t6yh8}
In the notation of Section \ref{n8d3wa}, let
\begin{gather*}
L=\bigl\{x\in S^\infty:\lim_nf_n(x)\text{ exists and is finite}\bigr\},
\\f(x)=\lim_nf_n(x)\text{ for each }x\in L\quad\text{and}\quad u^*=\sup_nu_n.
\end{gather*}
If $\sigma$ is the strategy \eqref{r6g7} then, for each $x\in L$,
\begin{gather*}
\sigma_n(x)\longrightarrow\delta_{f(x)}\,\text{ weakly if }\,u^*=u\,\text{ and}
\\\sigma_n(x)\longrightarrow\mathcal{S}\bigl(f(x),\,u-u^*\bigr)\,\text{ in total variation if }\,u^*<u.
\end{gather*}
Moreover, $P_\sigma(L)=1$ and $f\sim\mathcal{S}(0,u^*)$ under $P_\sigma$, namely
\begin{gather*}
P_\sigma(f\in A)=\mathcal{S}(0,u^*)(A)\quad\quad\text{for all }A\in\mathcal{B}.
\end{gather*}
\end{thm}

\begin{proof}
For all $n\ge 0$ and $x\in S^n$, the characteristic function of $\sigma_n(x)$ is
\begin{gather*}
\phi_n(x,t)=\int\exp\bigl(i\,t\,y\bigr)\,\sigma_n(x)(dy)=\exp\Bigl(i\,t\,f_n(x) -\frac{u-u_n}{2}\,\abs{t}^\gamma\Bigr).
\end{gather*}
Since $\phi_n(x,\cdot)$ is integrable, $\sigma_n(x)$ is absolutely continuous (with respect to Lebesgue measure) with density
\begin{gather*}
h_n(x,y)=(1/2\pi)\,\int\exp(-i\,t\,y)\,\phi_n(x,t)\,dt\quad\text{for all }y\in\mathbb{R}.
\end{gather*}
Having noted this fact, fix $x\in L$. Then,
\begin{gather*}
\lim_n\phi_n(x,t)=\exp\Bigl(i\,t\,f(x) -\frac{u-u^*}{2}\,\abs{t}^\gamma\Bigr)\quad\quad\text{for each }t\in\mathbb{R}
\end{gather*}
or equivalently
\begin{gather*}
\sigma_n(x)\longrightarrow\mathcal{S}\bigl(f(x),\,u-u^*\bigr)\quad\text{weakly}
\end{gather*}
where $\mathcal{S}\bigl(f(x),\,0\bigr):=\delta_{f(x)}$. Suppose now that $u^*<u$. Then, $\mathcal{S}\bigl(f(x),\,u-u^*\bigr)$ is absolutely continuous with density
\begin{gather*}
h(x,y)=(1/2\pi)\,\int\exp(-i\,t\,y)\,\exp\bigl(i\,t\,f(x)-\frac{u-u^*}{2}\,\abs{t}^\gamma\bigr)\,dt\quad\text{for all }y\in\mathbb{R}.
\end{gather*}
Therefore, $h(x,y)=\lim_nh_n(x,y)$ for all $y\in\mathbb{R}$, and this in turn implies
\begin{gather*}
\lim_n\,\norm{\sigma_n(x)-\mathcal{S}\bigl(f(x),\,u-u^*\bigr)}=\lim_n\int\Bigl(h(x,y)-h_n(x,y)\Bigr)^+\,dy=0.
\end{gather*}

Next, by the convergence of types theorem (see e.g. \cite[p. 174]{BRE}), the set $L$ can be written as
\begin{gather*}
L=\bigl\{x\in S^\infty:\sigma_n(x)\text{ converges weakly}\bigr\}.
\end{gather*}
Hence, $P_\sigma(L)=1$ as $P_\sigma$ is c.i.d.

Finally, as noted in Example \ref{h7n3w}, $f_n\sim\mathcal{S}(0,u_n)$ under $P_\sigma$. It follows that
\begin{gather*}
E_{P_\sigma}\bigl\{\exp(i\,t\,f)\bigr\}=\lim_nE_{P_\sigma}\bigl\{\exp(i\,t\,f_n)\bigr\}=\lim_n\exp\Bigl(-\frac{u_n}{2}\,\abs{t}^\gamma\Bigr)=\exp\Bigl(-\frac{u^*}{2}\,\abs{t}^\gamma\Bigr)
\end{gather*}
for all $t\in\mathbb{R}$. Hence, $f\sim\mathcal{S}(0,u^*)$ under $P_\sigma$, and this concludes the proof.
\end{proof}

\vspace{0.2cm}

\section{Some hints for future work}

As claimed in the Introduction, the main goal of this paper is to introduce and investigate new strategies satisfying conditions (i)-(ii). This has been realized through the strategies of Sections \ref{d4j8r}-\ref{n8d3wa}. However, obviously, many other strategies satisfying (i)-(ii) could be taken into account. In addition, some aspects related to our work could be investigated. A (non-exhaustive) list of research topics is appended below.

\vspace{0.2cm}

\begin{itemize}

\item Usually, the available information at time $n$ is broader than the observed values of $X_1,\ldots,X_n$. Hence, $\mathcal{G}_n=\sigma(X_1,\ldots,X_n)$ could be replaced with some $\sigma$-field $\mathcal{G}_n^*\supset\mathcal{G}_n$. The notion of c.i.d. sequence can be referred to an arbitrary filtration $(\mathcal{G}_n^*)$; see \cite{BPR2004}. Hence, to replace $\mathcal{G}_n$ with $\mathcal{G}_n^*$ seems to be technically possible, even if it requires a certain effort.

\vspace{0.2cm}

\item In \cite{HMW}, condition (ii) has been realized exploiting copulas. This approach looks very promising and it would be interesting to investigate its connections with \cite{BDPR2021} and this paper. Another recent reference to be involved is \cite{FHW2021}.

\vspace{0.2cm}

\item The non-standard approach to prediction (i.e., NSA) is quite natural as regards species sampling sequences. Hence, in the spirit of \cite{BCL}, the strategies of Sections \ref{d4j8r}-\ref{n8d3wa} (and more generally any other strategy satisfying (i)-(ii)) could be used in the species sampling framework. Among other things, one could investigate the length of the partition induced by the data when the strategy comes from Sections \ref{d4j8r}-\ref{n8d3wa}; see \cite{BCL} again.

\vspace{0.2cm}

\item As noted after Theorem \ref{u9mh1}, the properties of $\mu$ in Example \ref{m9e} are an open problem.

\vspace{0.2cm}

\item A relevant issue, deliberately left out of this paper, is the choice among different strategies. A possible approach is using scoring rules, as highlighted in \cite{GR2007}. For instance, in Section \ref{n8d3wa}, the choice of $\gamma$ and $u,u_1,u_2,\ldots$ could be made via scoring rules; see also Example \ref{sd5rf7l}.

\end{itemize}

\end{document}